\documentclass[11pt]{Mytran-l}
\usepackage[ansinew]{inputenc}
\usepackage[centertags]{amsmath}
\usepackage{exscale}
\usepackage{relsize}
\usepackage{amsfonts}
\usepackage{amssymb}
\usepackage{amsthm}
\usepackage{newlfont}
\usepackage{color}
\usepackage[colorlinks,citecolor=blue,urlcolor=blue,breaklinks=true]{hyperref}

\mathchardef\ordinarycolon\mathcode`\:
  \mathcode`\:=\string"8000
  \begingroup \catcode`\:=\active
    \gdef:{\mathrel{\mathop\ordinarycolon}}
  \endgroup

\def\R{{\Bbb R}}
\newtheorem*{mthm*}{Main Theorem}

\newtheorem{thm}{Theorem}

\newtheorem{prop}[thm]{Proposition}
\newtheorem{Def}[thm]{Definition}
\newtheorem{rem}[thm]{Remark}
\newtheorem{Cor}[thm]{Corollary}

\newtheorem{example}{Example}

\newcommand{\D}{\mathcal{D}}
\newcommand{\B}{\mathcal{B}}
\newcommand{\aut}{\emph{Aut}}

\newcommand{\Out}{\mbox{Out}}

\begin{document}

\title[Designs in Affine Spaces]{Block-Transitive Designs in Affine Spaces}

\author{Michael Huber}

\subjclass[2000]{Primary 51E10; Secondary 05B05, 20B25}

\keywords{Combinatorial design, finite affine space, block-transitive group of automorphisms, triply transitive permutation group, doubly homogeneous permutation group}

\date{June 24, 2009; and in revised form September 30, 2009}

\thanks{To Spyros Magliveras on the occasion of his 70th birthday.}

\maketitle
\vspace*{-0.5cm}

\begin{center}
{\footnotesize Wilhelm Schickard Institute for Computer Science\\
University of Tuebingen\\
Sand~13, D-72076 Tuebingen, Germany\\
E-mail: {\texttt{michael.huber@uni-tuebingen.de}}}
\end{center}

\smallskip

\begin{abstract}
This paper deals with block-transitive \mbox{$t$-$(v,k,\lambda)$} designs in affine spaces for large $t$, with a focus on the important index $\lambda=1$ case. We prove that there are no non-trivial \mbox{$5$-$(v,k,1)$} designs admitting a block-transitive group of automorphisms that is of affine type. Moreover, we show that the corresponding non-existence result holds for \mbox{$4$-$(v,k,1)$} designs, except possibly when the group is one-dimensional affine. Our approach involves a consideration of the finite $2$-homogeneous affine permutation groups.
\end{abstract}

\smallskip


\section{Introduction}\label{intro}

The construction and characterization of block-transitive \mbox{$t$-$(v,k,\lambda)$} designs in affine spaces is an interesting and beautiful topic of research. The situation when $t=2$, in particular for the index $\lambda=1$ case, has been studied in greater detail over the last decades. However,
less is known when $t\geq 3$. Obvious natural examples exist for $t=3$ and arbitrary $\lambda$, by using point \mbox{$3$-transitive} affine groups over the field $GF(2)$ as groups of automorphisms. For general \mbox{$t$-designs}, it has been shown that block-transitivity implies point \mbox{$2$-homogeneity} (and hence point-primitivity) when $t\geq 4$, while for $t<4$ an infinite number of counter-examples demonstrate that block-transitivity does not necessarily imply point-primitivity (see~Proposition~\ref{flag3hom}; and~\cite{DelDoy1989} for the case $t<4$).

Alltop~\cite{Alltop1971} constructed in 1971 the first explicit example of a block-transitive \mbox{$5$-design} in affine space, having $v=256$ points and block-size $k=24$. He showed that an orbit of a \mbox{$3$-transitive} affine group over $GF(2)$ on the \mbox{$k$-subsets} of the underlying vector space is a \mbox{$5$-design} whenever it is a \mbox{$4$-design}, and derived a necessary and sufficient condition for this to take place.
Alltop's construction has been extended by Cameron \& Praeger~\cite{CamPrae1993}, yielding a flag-transitive \mbox{$5$-$(2^8,24,\lambda)$} design with $\lambda=2^{21}{.}3^2{.}5^2{.}7{.}31$. Moreover, Cameron and  Praeger proved the non-existence of block-transitive \mbox{$t$-designs} for $t > 7$.

In this paper, we focus on block-transitive $4$- and \mbox{$5$-designs} in affine spaces for the important index $\lambda=1$ case. We will generalize several arguments developed in our earlier work on flag-transitive Steiner designs~(\cite{Hu2001,Hu2005,Hu2007,Hu2007a}, and~\cite{Hu2008} for a monograph) to the weaker condition of block-transitivity. Our approach involves a consideration of the finite \mbox{$2$-homogeneous} affine permutation groups. We remark that in~\cite{Hu_cam2009,Hu_block2009}, we already showed in particular that no block-transitive Steiner \mbox{$6$-design} or \mbox{$7$-design} admitting a \mbox{$3$-transitive} affine group over $GF(2)$ as a group of automorphisms can exist.

\smallskip

We prove the following main result:

\begin{mthm*}\label{mainthm}
There is no non-trivial Steiner \mbox{$5$-design} $\mathcal{D}$ admitting a block-transitive group \mbox{$G \leq \aut(\D)$}
of automorphisms that is of affine type. Moreover, there is no non-trivial Steiner \mbox{$4$-design} $\mathcal{D}$ admitting a
block-transitive group \mbox{$G \leq \aut(\D)$} of automorphisms that is of affine type, except possibly when $G \leq A \mathit{\Gamma} L(1,q)$.
\end{mthm*}

The paper is organized as follows. In Section~\ref{prelim}, we introduce the notation and preliminary results that are important for the remainder of the paper. A discussion on examples of block-transitive \mbox{$t$-designs} in affine spaces for $t \geq 3$ is followed by a proof of the non-existence of block-transitive Steiner $4$- and \mbox{$5$-designs} admitting a \mbox{$3$-transitive} affine group over $GF(2)$ as a group of automorphisms. In Section~\ref{des_afftype}, we investigate Steiner $4$- and \mbox{$5$-designs} with a group of affine type as a possibly block-transitive group of automorphisms. We may restrict here to finite \mbox{$2$-homogeneous} affine permutation groups. This investigation completes the proof of the Main Theorem.

\medskip


\section{Notation and Preliminaries}\label{prelim}

\begin{Def}\label{Stdes}\em
For positive integers $t \leq k \leq v$ and $\lambda$, a \mbox{\emph{$t$-$(v,k,\lambda)$ design}} $\mathcal{D}$
is a pair \mbox{$(X,\mathcal{B})$}, satisfying the following properties:

\begin{enumerate}
\item[(i)] $X$ is a set of $v$ elements, called \emph{points},

\smallskip

\item[(ii)] $\mathcal{B}$ is a family of \mbox{$k$-subsets} of $X$, called \emph{blocks},

\smallskip

\item[(iii)] every \mbox{$t$-subset} of $X$ is contained in exactly $\lambda$ blocks.

\end{enumerate}
\end{Def}

We denote points by lower-case and blocks by upper-case Latin
letters. Via convention, let $b:=\left| \mathcal{B} \right|$ denote the number of blocks.
A \emph{flag} of $\mathcal{D}$ is an incident point-block pair $(x,B)$, where $x \in X$ and $B \in \mathcal{B}$ with $x \in B$.
A \mbox{$t$-design} is called \emph{simple}, if the same \mbox{$k$-subset}
of points may not occur twice as a block. If not stated otherwise, we will restrict our attention to simple designs in this paper.
If $t<k<v$, then we speak of a \emph{non-trivial} \mbox{$t$-design}.
For historical reasons, a \mbox{$t$-$(v,k,\lambda)$ design} with index
$\lambda =1$ is called a \emph{Steiner \mbox{$t$-design}} (sometimes
also a \emph{Steiner system}). There are many infinite classes of
Steiner \mbox{$t$-designs} for $t=2$ and $3$, however for $t=4$ and
$5$ only a finite number are known. For a detailed treatment of
combinatorial designs, we refer
to~\cite{BJL1999,crc06,hall86,hupi85,stin04}. In
particular,~\cite{BJL1999,crc06} provide encyclopedic accounts of
key results and contain existence tables with known parameter sets.

In this paper, we are investigating \mbox{$t$-designs} which admit
groups of automorphisms with homogeneity properties
such as transitivity on blocks or flags. We consider
automorphisms of a \mbox{$t$-design} $\mathcal{D}$ as
permutations on $X$ which preserve $\mathcal{B}$, and
call a group \mbox{$G \leq \mbox{Aut} (\mathcal{D})$} of
automorphisms of $\mathcal{D}$ \emph{block-transitive} (respectively
\emph{flag-transitive}, \emph{point \mbox{$t$-transitive}},
\emph{point \mbox{$t$-homogeneous}}) if $G$ acts transitively on the
blocks (respectively transitively on the flags,
\mbox{$t$-transitively} on the points, \mbox{$t$-homogeneously} on
the points) of $\mathcal{D}$. For short, $\mathcal{D}$ is said to
be, e.g., block-transitive if $\mathcal{D}$ admits a
block-transitive group of automorphisms.

For \mbox{$\mathcal{D}=(X,\mathcal{B})$} a Steiner
\mbox{$t$-design} with \mbox{$G \leq \mbox{Aut} (\mathcal{D})$}, let
$G_x$ denote the stabilizer of a point $x \in X$, and $G_B$ the
setwise stabilizer of a block $B \in \mathcal{B}$. For $x, y \in X$,
we define $G_{xy}= G_x \cap G_y$.

For any $x \in \R$, let $\lfloor x \rfloor$ denote the greatest
positive integer which is at most $x$.

\smallskip

We state some basic combinatorial facts (see, for instance,~\cite{BJL1999}):

\begin{prop}\label{s-design}
Let $\mathcal{D}=(X,\mathcal{B})$ be a \mbox{$t$-$(v,k,\lambda)$}
design, and for a positive integer $s \leq t$, let $S \subseteq X$
with $\left|S\right|=s$. Then the total number of blocks incident
with each element of $S$ is given by
\[\lambda_s = \lambda \frac{{v-s \choose t-s}}{{k-s \choose t-s}}.\]
In particular, for $t\geq 2$, a \mbox{$t$-$(v,k,\lambda)$} design is
also an \mbox{$s$-$(v,k,\lambda_s)$} design.
\end{prop}

It is customary to set $r:= \lambda_1$ denoting the total
number of blocks incident with a given point.

\begin{Cor}\label{Comb_t=5}
Let $\mathcal{D}$ be a \mbox{$t$-$(v,k,\lambda)$}
design. Then the following holds:
\begin{enumerate}

\item[{(a)}] $bk = vr.$

\smallskip

\item[{(b)}] $\displaystyle{{v \choose t} \lambda = b {k \choose t}}.$

\smallskip

\item[{(c)}] $r(k-1)=\lambda_2(v-1)$ for $t \geq 2$.

\end{enumerate}
\end{Cor}

\begin{Cor}\label{Comb_t=5_divCond}
Let $\mathcal{D}$ be a \mbox{$t$-$(v,k,\lambda)$}
design. Then
\[\lambda {v-s \choose t-s} \equiv \, 0\; \emph{\bigg(mod}\;\, {k-s \choose t-s}\bigg)\]
for each positive integer $s \leq t$.
\end{Cor}

\begin{prop}{\em (\cite{Cam1976,Tits1964}).}\label{Cam}
If $\D$ is a non-trivial Steiner \mbox{$t$-design}, then
the following holds:
\begin{enumerate}

\item[(a)] $v\geq (t+1)(k-t+1).$

\smallskip

\item[(b)] $v-t+1 \geq (k-t+2)(k-t+1)$ for $t>2$. If equality
holds, then
\smallskip
$(t,k,v)=(3,4,8),(3,6,22),(3,12,112),(4,7,23)$, or $(5,8,24)$.
\end{enumerate}
\end{prop}

As we are in particular interested in the cases when $t=4$ or $5$,
we obtain from~(b) the following upper bound for the positive integer $k$.

\begin{Cor}\label{Cameron_t=5}
Let $\D$ be a non-trivial Steiner \mbox{$t$-design} with
\mbox{$t=4+i$}, where \mbox{$i=0,1$}. Then
\[k \leq \Bigl\lfloor \sqrt{v-\textstyle{\big(\frac{11}{4}}+i\big)} + \textstyle{\frac{5}{2}+i} \Bigr\rfloor.\]
\end{Cor}

\begin{rem}\label{equa_t=5}
\emph{If \mbox{$G \leq \mbox{Aut}(\mathcal{D})$} acts
block-transitively on any non-trivial Steiner \mbox{$t$-design} $\mathcal{D}$
with $t\geq 2$, then $G$ acts point transitively on $\mathcal{D}$ by a result of Block~\cite[Thm.\,2]{Block1965}. In view of Corollary~\ref{Comb_t=5}~(b), this gives the equation
\[b=\frac{{v \choose t}}{{k \choose
t}}=\frac{v\left|G_{x}\right|}{\left| G_B \right|},\] where
$x$ is a point in $X$ and $B$ a block in
$\mathcal{B}$.}
\end{rem}

We also state a generalization of Block's result, which is due to Cameron \& Praeger~\cite[Thm.\,2.1]{CamPrae1993}.

\begin{prop}{\em (Cameron \& Praeger,~1993).}\label{flag3hom}
Let $\mathcal{D}$ be a \mbox{$t$-$(v,k,\lambda)$} design with $t\geq 2$. Then, the following holds:

\begin{enumerate}

\item[{(a)}] If \mbox{$G \leq \emph{Aut}(\mathcal{D})$} acts block-transitively on $\mathcal{D}$,
then $G$ also acts point \linebreak \mbox{$\lfloor t/2
\rfloor$-homogeneously} on $\mathcal{D}$.

\smallskip

\item[{(b)}] If \mbox{$G \leq \emph{Aut}(\mathcal{D})$} acts flag-transitively on $\mathcal{D}$,
then $G$ also acts point \linebreak \mbox{$\lfloor (t+1)/2
\rfloor$-homogeneously} on $\mathcal{D}$.

\end{enumerate}
\end{prop}

\medskip


\section{Block-Transitive Designs and Triply Transitive Affine Linear Groups}\label{des_AGL}

In this section we discuss block-transitive designs which admit a $d$-dim- ensional affine group $G=AGL(d,2)$ in its triply transitive action on the $2^d$ points of the underlying vector space $V=V(d,2)$. As $G$ is \mbox{$3$-transitive}, clearly for every cardinality $k$, every orbit of $G$ on $k$-subsets of $V$ yields a \mbox{$3$-design}. Alltop~\cite{Alltop1971} showed that such an orbit is a \mbox{$5$-design} whenever it is a \mbox{$4$-design}, and derived a necessary and sufficient condition for this to take place. He constructed this way the first block-transitive $t$-design in affine space with $t>3$.

\begin{example}{\em (Alltop,~1971).}\label{alltop}
There exists a \mbox{$5$-$(2^8,24,\lambda)$} design $\mathcal{D}$ admitting a block-transitive group \mbox{$G \leq \emph{Aut}(\mathcal{D})$} with \mbox{$G=AGL(8,2)$} (where $\lambda=2^{21}{.}3^2{.}5^2{.}7{.}31$).
\end{example}

Alltop's construction has been extended by Cameron \& Praeger~\cite{CamPrae1993}, yielding a flag-transitive design.

\begin{example}{\em (Cameron \& Praeger,~1993).}\label{campraeg}
There exists a \mbox{$5$-$(2^8,24,\lambda)$} design $\mathcal{D}$ admitting a flag-transitive group \mbox{$G \leq \emph{Aut}(\mathcal{D})$} with \mbox{$G=AGL(8,2)$} (where $\lambda=2^{21}{.}3^2{.}5^2{.}7{.}31$).
\end{example}

\begin{rem}\em
Bierbrauer~\cite{Bierb2001} has constructed an infinite family of non-simple \mbox{$7$-designs} which are invariant under $AGL(d,2)$, but not block-transitive.
\end{rem}

Cameron \& Praeger~\cite{CamPrae1993} proved the non-existence of block-transitive \linebreak \mbox{$t$-designs} for $t > 7$.

\begin{thm}{\em (Cameron \& Praeger,~1993).}
Let $\mathcal{D}$ be a non-trivial \mbox{$t$-design}. If \mbox{$G \leq \emph{Aut}(\mathcal{D})$} acts
block-transitively on $\mathcal{D}$ then $t \leq 7$, while if \mbox{$G \leq \emph{Aut}(\mathcal{D})$}
acts flag-transitively on $\mathcal{D}$ then $t \leq 6$.
\end{thm}

Considering the index $\lambda=1$ case, we have shown in~\cite{Hu_cam2009,Hu_block2009} in particular that there exists no non-trivial Steiner \mbox{$6$-design} or \mbox{$7$-design} $\mathcal{D}$ admitting a block-transitive group \mbox{$G \leq \mbox{Aut}(\mathcal{D})$} with \mbox{$G=AGL(d,2)$}, $v=2^d, d\geq 3$.

\smallskip

In this section, we prove:

\begin{prop}\label{AGL(d,2)}
There is no non-trivial Steiner \mbox{$4$-design} or \mbox{$5$-design} $\mathcal{D}$ admitting a block-transitive group \mbox{$G \leq \emph{Aut}(\mathcal{D})$}, where \mbox{$G=AGL(d,2)$}, $v=2^d, d\geq 3$.
\end{prop}

\begin{proof}
As trivial designs are excluded, let $v=2^d > k > t$ for $t=4$ and $5$, respectively.
Furthermore, we may assume that always $d>3$ in view of Corollary~\ref{Cameron_t=5}.
First, we consider the case when $t=4$. Any three distinct points being non-collinear in
$AG(d,2)$, they generate an affine plane. Let $\mathcal{E}$ be the
\mbox{$2$-dimensional} vector subspace spanned by the first two basis vectors $e_1$ and $e_2$
of the vector space $V=V(d,2)$. Then the pointwise stabilizer of $\mathcal{E}$ in $SL(d,2)$ (and therefore also in $G$)
acts point-transitively on \mbox{$V \setminus \mathcal{E}$}.
If the unique block $B \in \B$ which is incident with the \mbox{$4$-subset}
\mbox{$\{0,e_1,e_2,e_1+e_2\}$} contains some point outside
$\mathcal{E}$, then $B$ contains all points of \mbox{$V \setminus \mathcal{E}$},
and so $k \geq 2^d=v$, which is impossible. Hence $B$ can be identified with $\mathcal{E}$,
and by the block-transitivity of $G$, each block must be an
affine plane. This implies that always $k=4$, a contradiction.

Now, let $t=5$. Any five distinct points being non-coplanar in
$AG(d,2)$, they generate an affine subspace of dimension at
least $3$. Let $\mathcal{E}$ be the
\mbox{$3$-dimensional} vector subspace spanned by the first three basis vectors $e_1, e_2, e_3$ of $V$.
Then the pointwise stabilizer of $\mathcal{E}$ in $SL(d,2)$ (and therefore also in $G$)
acts point-transitively on \mbox{$V \setminus \mathcal{E}$}.
If the unique block $B \in \B$ which is incident with the \mbox{$5$-subset}
\mbox{$\{0,e_1,e_2,e_3,e_1+e_2\}$} contains some point outside
$\mathcal{E}$, then $B$ contains all points of \mbox{$V \setminus \mathcal{E}$},
and so \mbox{$k \geq 2^d-3$}, a contradiction to
Corollary~\ref{Cameron_t=5}. The block-transitivity
of $G$ now implies that each block must be contained in a
\mbox{$3$-dimensional} affine subspace. This leads to a contradiction as any
five distinct points that generate a \mbox{$4$-dimensional} affine subspace must also be incident with a unique block by Definition~\ref{Stdes}.
\end{proof}

\medskip


\section{\mbox{Block-Transitive Designs and Further Groups of Affine Type}}\label{des_afftype}

We investigate in this section Steiner $4$- and \mbox{$5$-designs} with a group of affine type as a possibly block-transitive group of automorphisms.  We note that due to Proposition~\ref{flag3hom}, we may restrict ourselves to the finite \mbox{$2$-homogeneous} affine permutation groups.

Let $G$ be a group acting \mbox{$2$-homogeneously} on a finite set $X$ of $v \geq 3$ points. If $G$ is not \mbox{$2$-transitive} on $X$, then
\mbox{$G \leq A \mathit{\Gamma} L(1,q)$} with \mbox{$v=q \equiv 3$ (mod $4$)} by a result of Kantor~\cite{Kant1969}.
On the other hand, relying on the classification of the finite simple groups, all \mbox{$2$-transitive} groups on $X$ are known (cf.~\cite{CSK1976,Gor1982,Her1974,Her1985,Hup1957,Kant1985,Lieb1987,Mail1895}). By a classical result of Burnside, they split into two types of groups. In the context of our consideration, we will deal with the groups of affine type: A finite \mbox{$2$-transitive} permutation group on $X$ is called of \emph{affine type}, if $G$ contains a regular normal subgroup $T$ which is elementary Abelian of order $v=p^d$, where $p$ is a prime.
\mbox{If $a$ divides $d$,} and if we identify $G$ with a group of affine transformations
\[x \mapsto x^g+u\]
of $V=V(d,p)$, where $g \in G_0$ and $u \in V$, then particularly one of the following occurs:

\begin{enumerate}

\smallskip

\item[(1)] $G \leq A \mathit{\Gamma} L(1,p^d)$

\smallskip

\item[(2)] $G_0 \unrhd SL(\frac{d}{a},p^a)$, $d \geq 2a$

\smallskip

\item[(3)] $G_0 \unrhd Sp(\frac{2d}{a},p^a)$, $d \geq 2a$

\smallskip

\item[(4)] $G_0 \unrhd G_2(2^a)'$, $d=6a$

\smallskip

\item[(5)] $G_0 \cong A_6$ or $A_7$, $v=2^4$

\smallskip

\item[(6)] $G_0 \unrhd SL(2,3)$ or $SL(2,5)$, $v=p^2$,
$p=5,7,11,19,23,29$ or $59$, or $v=3^4$

\smallskip

\item[(7)] $G_0$ contains a normal extraspecial subgroup $E$ of
order $2^5$, and $G_0/E$ is isomorphic to a subgroup of $S_5$,
$v=3^4$

\smallskip

\item[(8)] $G_0 \cong SL(2,13)$, $v=3^6,$
\end{enumerate}

\begin{prop}
There is no non-trivial Steiner \mbox{$5$-design} $\mathcal{D}$ admitting a block-transitive group \mbox{$G \leq \emph{Aut}(\mathcal{D})$}
with $G\leq A \mathit{\Gamma} L(1,p^d)$, $v=p^d$.
\end{prop}

\begin{proof}
Clearly, $\left| G \right| \big| \left|A \mathit{\Gamma} L(1,v)\right|=v(v-1)d$.
From Remark~\ref{equa_t=5} follows
\[(v-2)(v-3)(v-4) \left| G_{B} \right| \big| k(k-1)(k-2)(k-3)(k-4)d.\]
By Proposition~\ref{Cam}~(b), we have
\[v-4 \geq (k-3)(k-4).\]
Hence
\[(v-2)(v-3)\left|G_{B} \right| \leq k(k-1)(k-2)d.\]
However, as $d \leq \log_2v$, this is always impossible in view of Corollary~\ref{Cameron_t=5}.
\end{proof}

\begin{prop}\label{SL}
There is no non-trivial Steiner \mbox{$4$-design} or \mbox{$5$-design} $\mathcal{D}$ admitting a
block-transitive group \mbox{$G \leq \emph{Aut}(\mathcal{D})$} of affine type, where $G_0 \unrhd SL(\frac{d}{a},p^a)$, $d\geq 2a$.
\end{prop}

\begin{proof}
First, let $t=4$. The case $p^a=2$ has already been treated in Proposition~\ref{AGL(d,2)}.
For $p^a=3$, we may argue similarly. Hence, let us assume that $p^a >3$.
For $d=2a$, let $U=U(\text{\footnotesize{$\langle$}} e_1
\text{\footnotesize{$\rangle$}}) \leq G_0$ denote the subgroup of
all transvections with axis $\text{\footnotesize{$\langle$}} e_1
\text{\footnotesize{$\rangle$}}$. Obviously, $U$ fixes as points only the elements of
$\text{\footnotesize{$\langle$}} e_1 \text{\footnotesize{$\rangle$}}$. Thus, $G_0$ has point-orbits of
length at least $p^a$ outside $\text{\footnotesize{$\langle$}} e_1
\text{\footnotesize{$\rangle$}}$. Let \mbox{$S=\{0,e_1,x,y\}$} be a
\mbox{$4$-subset} of distinct points with $x,y \in
\text{\footnotesize{$\langle$}} e_1
\text{\footnotesize{$\rangle$}}$. Clearly, $U$ fixes the unique
block $B \in \B$ which is incident with $S$. Therefore, if $B$
contains at least one point outside $\text{\footnotesize{$\langle$}}
e_1 \text{\footnotesize{$\rangle$}}$, then $k \geq p^a +4$, a contradiction in
view of Corollary~\ref{Cameron_t=5}. Hence, $B$ is completely contained in
$\text{\footnotesize{$\langle$}} e_1
\text{\footnotesize{$\rangle$}}$. As $G$ is block-transitive,
we may conclude that each block lies in an affine line. However, by
Definition~\ref{Stdes}, any four distinct
non-collinear points must also be incident with a unique block, a
contradiction. Thus, let us assume that $d \geq 3a$.  Then
$SL(\frac{d}{a},p^a)_{e_1}$ (and hence also $G_{0,e_1}$) acts
point-transitively on $V \setminus \text{\footnotesize{$\langle$}}
e_1 \text{\footnotesize{$\rangle$}}$. As above, let
\mbox{$S=\{0,e_1,x,y\}$} be a \mbox{$4$-subset} of distinct points
with $x,y \in \text{\footnotesize{$\langle$}} e_1
\text{\footnotesize{$\rangle$}}$. If the unique block $B \in \B$
which is incident with $S$ contains some point outside
$\text{\footnotesize{$\langle$}} e_1
\text{\footnotesize{$\rangle$}}$, then $B$ contains all
points outside, and thus $k \geq p^d-p^a+4$, contradicting Corollary~\ref{Cameron_t=5}.
We conclude that $B$ lies
completely in $\text{\footnotesize{$\langle$}} e_1
\text{\footnotesize{$\rangle$}}$, and we can proceed with the same argument as above.
For $t=5$, our methods may be applied mutatis mutandis.
\end{proof}

\begin{prop}
There is no non-trivial Steiner \mbox{$4$-design} $\mathcal{D}$ admitting a
block-transitive group \mbox{$G \leq \emph{Aut}(\mathcal{D})$} of affine type, where $G_0 \unrhd Sp(\frac{2d}{a},p^a)$, $d \geq 2a$.
\end{prop}

\begin{proof}
We only consider in detail the case $t=4$. Our arguments work, mutatis mutandis, also for the case $t=5$.
First, let $p^a \neq 2.$ The permutation group
$PSp(\frac{2d}{a},p^a)$ on the points of the associated projective
space is a rank $3$ group, and the orbits of the one-point
stabilizer are known
(e.g.~\cite[Ch.\,II,\,Thm.\,9.15\,(b)]{HupI1967}). Thus, $G_0 \unrhd
Sp(\frac{2d}{a},p^a)$ has exactly two orbits on \mbox{$V \setminus
\text{\footnotesize{$\langle$}} x \text{\footnotesize{$\rangle$}}$}
$(0 \neq x \in V)$ of length at least
\[\frac{p^a(p^{2d-2a}-1)}{p^a-1}=\sum_{i=1}^{\frac{2d}{a}-2}p^{ia} >
p^d.\] Let $S=\{0,x,y,z\}$ be a \mbox{$4$-subset} with $y,z \in
\text{\footnotesize{$\langle$}} x
\text{\footnotesize{$\rangle$}}$. If the unique block which is
incident with $S$ contains at least one point of \mbox{$V
\setminus \text{\footnotesize{$\langle$}} x
\text{\footnotesize{$\rangle$}}$}, then $k
> p^d+4$, a contradiction to Corollary~\ref{Cameron_t=5}. Therefore, we may continue
with our argumentation as in Proposition~\ref{SL}.

Considering the case $p^a=2$, let $v=2^{2d} > k> 4$. For $d=2$ (here $Sp(4,2) \cong S_6$ as
well-known), Corollary~\ref{Cameron_t=5} implies that $k=5$ or $6$, each of which is not possible in view of Corollary~\ref{Comb_t=5}~(c).
Therefore, let $d>2$. It is easily seen that there are $2^{2d-1}(2^{2d} -
1)$ hyperbolic pairs in the non-degenerate symplectic space
$V=V(2d,2)$, and by Witt's theorem, $Sp(2d,2)$ is transitive on
these hyperbolic pairs. Let $\{x,y\}$ denote a hyperbolic pair, and
$\mathcal{E}=\text{\footnotesize{$\langle$}} x,y
\text{\footnotesize{$\rangle$}}$ the hyperbolic plane spanned by
$\{x,y\}$. As $\mathcal{E}$ is non-degenerate, we have the
orthogonal decomposition
\[V=\mathcal{E}\perp\mathcal{E}^\perp.\]
Clearly, $Sp(2d,2)_{\{x,y\}}$ stabilizes $\mathcal{E}^\perp$ as a
subspace, which implies that \linebreak \mbox{$Sp(2d,2)_{\{x,y\}} \cong
Sp(2d-2,2)$}. As $\Out(Sp(2d,2))=1$, we have therefore
\[Sp(2d-2,2) \cong Sp(2d,2)_{\{x,y\}} \unlhd Sp(2d,2)_\mathcal{E}=G_{0,\mathcal{E}}.\]
As $Sp(2d-2,2)$ acts transitively on the non-zero vectors of the
\mbox{$(2d-2)$}-dimensional symplectic subspace, it is easy to see
that the smallest orbit on \mbox{$V \setminus \mathcal{E}$} under
$G_{0,\mathcal{E}}$ has length at least $2^{2d-2}-1$. If the unique
block $B \in \B$ which is incident with the \mbox{$4$-subset}
\mbox{$\{0,x,y,x+y\}$} contains some point in \mbox{$V \setminus
\mathcal{E}$}, then $k \geq 2^{2d-2}+3$, which is impossible in view of
Corollary~\ref{Cameron_t=5}. Thus, $B$ can be
identified with $\mathcal{E}$, leading again to a contradiction.
\end{proof}

\begin{prop}
There is no non-trivial Steiner \mbox{$4$-design} or \mbox{$5$-design} $\mathcal{D}$ admitting a
block-transitive group \mbox{$G \leq \emph{Aut}(\mathcal{D})$} of affine type, where $G_0 \unrhd G_2(2^a)'$, $d=6a$.
\end{prop}

\begin{proof}
Let $t=4$. For $t=5$, we may argue mutatis mutandis.
First, let $a=1$. Then $v=2^6=64$, and so $k \leq 10$ by Corollary~\ref{Cameron_t=5}.
We have $\left| G_2(2)' \right|=2^5 \cdot 3^3 \cdot 7$ and $\left| \Out(G_2(2)') \right| =2$.
Using Corollary~\ref{Comb_t=5_divCond} and Remark~\ref{equa_t=5}, we can easily rule out the possibilities for $k$.
Now, let $a>1$. As here $G_2(2^a)$ is simple non-Abelian, it is
sufficient to consider $G_0 \unrhd G_2(2^a)$. The permutation group
$G_2(2^a)$ is of rank $4$, and for $0 \neq x \in V$, the one-point
stabilizer $G_2(2^a)_x$ has exactly three orbits $\mathcal{O}_i$
$(i=1,2,3)$ on \mbox{$V \setminus \text{\footnotesize{$\langle$}} x
\text{\footnotesize{$\rangle$}}$} of length
$2^{3a}-2^a,2^{5a}-2^{3a},2^{6a}-2^{5a}$ (cf., e.g.,~\cite{Asch1987}
or~\cite[Thm.\,3.1]{CaKa1979}). Thus, $G_0$ has exactly three orbits
on \mbox{$V \setminus \text{\footnotesize{$\langle$}} x
\text{\footnotesize{$\rangle$}}$} of length at least $\left|
\mathcal{O}_i \right|.$ Let $S=\{0,x,y,z\}$ be a \mbox{$4$-subset}
with $y,z \in \text{\footnotesize{$\langle$}} x
\text{\footnotesize{$\rangle$}}$. Again, we will show that the
unique block $B \in \B$ which is incident with $S$ lies completely
in $\text{\footnotesize{$\langle$}}x
\text{\footnotesize{$\rangle$}}$. If $B$ contains at least one point
of \mbox{$V \setminus \text{\footnotesize{$\langle$}} x
\text{\footnotesize{$\rangle$}}$} in $\mathcal{O}_2$ or
$\mathcal{O}_3$, then we obtain again a contradiction to
Corollary~\ref{Cameron_t=5}. Thus, we only have to consider the case
when $B$ contains points of \mbox{$V \setminus
\text{\footnotesize{$\langle$}} x \text{\footnotesize{$\rangle$}}$}
which all lie in $\mathcal{O}_1$. By~\cite{Asch1987}, the orbit
$\mathcal{O}_1$ is exactly known, and we have
\[\mathcal{O}_1 = x\Delta \setminus
\text{\footnotesize{$\langle$}} x \text{\footnotesize{$\rangle$}},\]
where $x\Delta = \{y \in V \mid f(x,y,z)=0 \;\,\mbox{for all} \;\, z
\in V\}$ with an alternating trilinear form $f$ on $V$. Then $B$
consists, apart from elements of $\text{\footnotesize{$\langle$}} x
\text{\footnotesize{$\rangle$}}$, exactly of $\mathcal{O}_1$. Since
$\left| \mathcal{O}_1 \right| \neq 1$, we can choose
$\text{\footnotesize{$\langle$}} \overline{x}
\text{\footnotesize{$\rangle$}} \in x\Delta$ with
$\text{\footnotesize{$\langle$}} \overline{x}
\text{\footnotesize{$\rangle$}} \neq \text{\footnotesize{$\langle$}}
x \text{\footnotesize{$\rangle$}}$. However, for symmetric reasons,
the \mbox{$4$-subset} $\{0,\overline{x},\overline{y},\overline{z}\}$
with $\overline{y},\overline{z} \in \text{\footnotesize{$\langle$}}
\overline{x} \text{\footnotesize{$\rangle$}}$ must also be incident
with the unique block $B$, a contradiction to the fact that
$\overline{x} \Delta \neq x \Delta$ for
$\text{\footnotesize{$\langle$}} \overline{x}
\text{\footnotesize{$\rangle$}} \neq \text{\footnotesize{$\langle$}}
x \text{\footnotesize{$\rangle$}}$. Consequently, $B$ is completely contained
in $\text{\footnotesize{$\langle$}} x
\text{\footnotesize{$\rangle$}}$, and we may argue as in the Propositions
above.
\end{proof}

\begin{prop}
There is no non-trivial Steiner \mbox{$4$-design} or \mbox{$5$-design} $\mathcal{D}$ admitting a
block-transitive group \mbox{$G \leq \emph{Aut}(\mathcal{D})$} of affine type, where $G_0$ is as in Cases~$(5)-(8)$.
\end{prop}

\begin{proof}
We have in these cases only finitely many possibilities for $k$ to check,
which can easily be ruled out by hand, combining Corollaries~\ref{Comb_t=5},~\ref{Comb_t=5_divCond},~\ref{Cameron_t=5},
and Remark~\ref{equa_t=5}.
\end{proof}

\medskip


\subsection*{Acknowledgment}\hfill

\smallskip

The author gratefully acknowledges support by the Deutsche Forschungsgemeinschaft (DFG) via a Heisenberg grant (Hu954/4).

\bibliographystyle{amsplain}
\bibliography{XbibDesAff}
\end{document}